\documentclass[reqno,12pt,a4paper,oneside]{amsart}

\usepackage{amsmath}
\usepackage{array}
\usepackage{mathrsfs}
\usepackage{amssymb}
\usepackage{amsfonts}
\usepackage{latexsym}
\usepackage{amsthm}
\usepackage{amsmath}
\usepackage[arrow, matrix, curve]{xy}
\usepackage{graphicx}	
\usepackage{tikz}
\usetikzlibrary{matrix,arrows}
\input xy

\theoremstyle{plain}
\newtheorem{theorem}{\bf Theorem}[section]
\newtheorem{lemma}[theorem]{\bf Lemma}
\newtheorem{corollary}[theorem]{\bf Corollary}
\newtheorem{proposition}[theorem]{\bf Proposition}

\def\a{\alpha}  \def\cA{{\mathcal A}}       
\def\d{\delta}         
\def\k{\kappa}         
\def\l{\lambda}        
\def\n{\nu}            
\def\r{\rho}           
\def\o{\omega}         

\def\Z{{\mathbb Z}}    
 \def\N{{\mathbb N}}  

\def\no{\noindent}

\newcommand{\mg}{\overline{\mathcal{M}}_g}
\newcommand{\mgnn}{\overline{\mathcal{M}}_{g,2n}}
\newcommand{\mgn}{\overline{\mathcal{M}}_{g,n}}
\newcommand{\mgg}{\overline{\mathcal{M}}_{g,g}}

\newcommand{\mgm}{\overline{\mathcal{M}}_{g,m}}

\newcommand{\ngn}{\overline{\mathcal{N}}_{g,n}}
\newcommand{\M}{\overline{\mathcal{M}}}
\newcommand{\mG}{\overline{\mathcal{M}}^G}
\newcommand{\mH}{\overline{\mathcal{M}}^H}
\newcommand{\mGgn}{\overline{\mathcal{M}}^G_{g,n}}
\newcommand{\mS}{\overline{\mathcal{M}}^{S_n}}

\newcommand{\tmG}{\tilde{\mathcal{M}}^G}
\newcommand{\dirr}{\d_{\mathrm{irr}}} 

\def\Pic{\mathop{\mathrm{Pic }}}
\def\Sym{\mbox{Sym}}

\let\geq\geqslant
\let\leq\leqslant
\def\eq{\begin{equation}}
\def\qe{\end{equation}}

\newcommand\blfootnote[1]{%
  \begingroup
  \renewcommand\thefootnote{}\footnote{#1}%
  \addtocounter{footnote}{-1}%
  \endgroup
}

\begin{document}

\title{On quotients of $\mgn$ by certain subgroups of $S_n$}

\author{Irene Schwarz}
\address{Humboldt Universit\"at Berlin, Institut f\"ur Mathematik, 
Rudower Chausee 25, 12489 Berlin, Germany}
\email{schwarzi@math.hu-berlin.de}
\begin{abstract}
\no We show that certain quotients of the compactified moduli space  of 
$n-$ pointed genus $g$ curves, $\mG:= \mgn / G$, are of general type, for a fairly broad class of subgroups $G$ of the symmetric group  $S_n$ which act by permuting the $n$ marked points. The values of $(g,n)$ which we specify in our theorems are near optimal in the sense that, at least in he cases that G is the full symmetric group $S_n$ or a product $S_{n_1}\times \ldots \times S_{n_m}$, there is a relatively narrow transitional zone in which $\mG$ changes its behaviour from being of general type to its opposite, e.g. being uniruled or even unirational. 
As an application we consider the universal difference variety $\mgnn /S_n \times S_n$.


  
\end{abstract}

\thanks{This paper is part of my PhD written under the supervision of G. Farkas at Humboldt Universit\"at Berlin, Institut f\"ur Mathematik.  I wish to thank my advisor for his efficient support. 
Furthermore, I acknowledge helpful comments of Will Sevin on an earlier version of \cite{s2} which also led to improvements in this paper.
}


\maketitle

\setcounter{page}{1}

\section{introduction}

In this paper we shall consider a class of quotients of $\mgn$, the compactified moduli
space of $n$-pointed genus g complex curves, by certain subgroups of the symmetric group $S_n$ which act 
by permuting the marked points. Our aim is to analyse under which conditions such quotients are of general type or, in a complementary case, uniruled or even unirational. As usual, we do this by using the Kodaira dimension.

We were led to considering the questions adressed in the present paper
by analyzing  an analogous problem for the compactified moduli space $\ngn$
of $n-$ nodal genus $g$ curves in \cite{s2}. Here $\ngn = \mgnn/G$ where the special group $G$ is also a subgroup of $S_{2n}$, namely the semidirect product
$G:=(\Z_2)^n\rtimes S_n$. In view of the great importance of $n-$nodal curves,
e.g. in the deformation type arguments used in the proof of the Brill-Noether theorem, 
this problem was directly motivated by geometry.
 
Our  proof in \cite{s2}, however, let us realize that  there are some related results for {\em general} quotients of $\mgn$ which in some aspect are {\em different} from the special case of $\ngn$. In particular, it is important  that $G:=(\Z_2)^n\rtimes S_n$ is a semidirect product and
{\em not} a product of subgroups. The main point of this paper is to prove first results in this direction for a class of general quotients.

For  $G \subset S_n$,  we denote the quotient by this action as $ \mgn^G := \mgn / G$ and we suppress the subscript $(g,n)$ if we feel it unnecessary  within  a given context.  Then the natural quotients induce the chain of morphisms of schemes
\begin{equation}  \label{quotient}
\mgn \to \mgn^G  \to \mgn^{S_n}  
\end{equation} 
which by standard arguments (weak additivity of the Kodaira dimension for base and fibre) gives the following ordering for the Kodaira dimension
that 
\begin{equation}
\k(\mgn) \geq \k(\mgn^G)  \geq \k(\mgn^{S_n}).
\end{equation}
 Since all algebraic varieties in \eqref{quotient} have the same dimension, one gets 
 
 \begin{equation}
 \mgn \quad \mbox{of general type    } \Rightarrow \mgn^G \quad \mbox{of g. t.   }\Rightarrow \mgn^{S_n} \quad \mbox{ of g.t.}
 \end{equation}
By the same argument, one gets
\begin{lemma}  \label{sub}
For any subgroup $H$ of $G$ one has
\begin{equation}
\mG \quad \mbox{of general type    } \Rightarrow \mH \quad \mbox{of general type    }
\end{equation}
\end{lemma}

We shall need the ramification divisor $R$ (see equation \eqref{R} below) of the quotient map $ \pi: \mgn \to \mG$. Denoting by $(i \  j)$ the transposition in $S_n$ interchanging $i$ and $j$ and checking where sheets will come together,  one readily finds
\begin{equation}
R = \sum_{(i \  j) \in G} \d_{0, \{i,j\}}
\end{equation}
where the (standard) definition of the boundary divisor $\d_{0, \{i,j\}}$ is given in Section \ref{prem} below, which in particular introduces all divisors needed in this paper.
Then the well known explicit formula for the canonical divisor $K_{\mgn}$ gives

\begin{corollary} \label{corr}
The pullback $K_G:= \pi^*(K_{\mG})$ to $\mgn$ is given by
\begin{equation}    \label{kg}
K_G = K_{\mgn} - R = 13 \l + \psi - 2 \d - \sum_{(i \ j) \in G} \d_{0, \{i,j\}}
\end{equation}
\end{corollary}

For the sake of the reader, we shall now recall   

\begin{proposition} \label{proposit}
The moduli space $\mgn$ is of general type for $g\geq 23$ or for $n\geq n_{\mathrm{min}}(g)$ given in the following table:
\begin{table}[h!]
$$\begin{array}{c|c|c|c|c|c|c|c|  c|c|c|c|c|  c|c|c|c|c|  c|c}
g &4&5&6&7&8&9&10&11&12&13&14&15&16&17&18&19&20&21&22\\
\hline
n_{\mathrm{min}}&16&15&16&15&14&13&11&12&11&11&10&10&9&9&9&7&6&4&4
\end{array}$$
\caption{ } \label{table mg}
\end{table}
\end{proposition}

This collects results from \cite{l1, f, fv4}  and covers all  cases known up to now.
As a corollary, one then obtains 
\begin{theorem} \label{gen}
In all cases of Proposition \ref{proposit}, if $G$ does not contain any transposition, then $\mG$ also is of general type.
\end{theorem}

 \begin{proof}
By Corollary \ref{corr}, the divisor classes $K_G$ and $K_{\mgn}$ coincide. The moduli space $\mgn$ is of general type if and only if  $K_{\mgn}$ is the sum of an ample and an effective divisor and, similarly, 
$ \mgn^G $ is of general type if and only if $K_G$ is the sum of an ample and an effective divisor which are both also $G-$ invariant. But in all cases of Table 1 only
$S_n-$ invariant divisors have been used to show that $\mgn$ is of general type, see
\cite{l1, f, fv4}. This proves our claim.
 \end{proof}
 
In particular, this covers the case where $G$ is cyclic (and different from $\Z_2$) or the  cardinality $|G|$ is odd. It 
also covers the largest non-trivial subgroup of $S_n$, the alternating group $A_n=\mbox{kernel} \mbox{  signum} $ (in fact G contains no transpositions, if and only if it is a subgroup of $A_n$).  This has an obvious geometric interpretation: The set of $n- $ tuples of $n$ fixed different points on a genus g curve carries 
a notion of {\em orientation}: Two $n-$tuples have the same orientation, if they are mapped one to another by an even permutation. Taking the quotient by $S_n$ corresponds to passing from $n-$pointed curves to $n-$marked curves, while taking the quotient by $\cA_n$ means passing to curves marked in $n$ points with orientation. Under the first action the property of being of general type might change, but it is invariant under the second.

We remark that it is very natural to use $S_n-$ invariant divisors. We expect that it is possible to do so in {\em any} case in which $\mgn$ is of general type. 
Thus we conjecture:  If $\mgn$ is of general type, then $\mG$ is of general type for any subgroup  $G$ of $A_n$.

If the subgroup $G$ does contain transpositions, 
the situation is more complicated. For large $g$, the following theorem contains an (easy) general result, while for small $g$   we shall need explicit computations with well chosen divisors.   
Collecting  from \cite{fv2}  and supplementing this by an analog proof in the cases $g=22,23$
(using the divisors of \cite{fv2})  we obtain

\begin{proposition} \label{one}
The space $\mS$ (and thus, by Lemma \ref{sub}, the space $\mG$ for any subgroup $G$ of $S_n$) is of general type if
\begin{enumerate}
\item[(i)]  $g \geq 24, \quad n<g, \quad$     or
\item[(ii)]  $13 \leq g \leq 23,$   and   $ n_{\mathrm{min}}(g) \leq n \leq g-1,$ where $n_{\mathrm{min}}(g)$ is given in the following table
\end{enumerate}

\begin{table}[h!]
$$\begin{array}{c|c|c|c|c|  c|c|c|c|c|  c|c|c}
g &12&13&14&15&16&17&18&19&20&21&22&23\\
\hline
n_{\mathrm{min}}&10&11&10&10&9&9&10&7&6&4&7&1
\end{array}$$
\caption{ } \label{table mS}
\end{table}
Furthermore, 
the domain of values $(g,n)$ for which $\mS$ is of general type is {\em near optimal} since it is  known that
\begin{enumerate}
\item uniruled, if $n>g$ (for any $g$) or $g \in \{10,11\}$ with $n \neq g$
\item  unirational, if $ g<10, n \leq g$
\item  for $g \geq 12$ the Kodaira dimension $\k(\mgg)=3g -3$ is intermediary.
\end{enumerate}
\end{proposition}

Here the result of (i) follows from weak additivity of the Kodaira dimension, while (ii) is proven in \cite{fv2} by explicit computation. The first assertion in (1) follows from Riemann-Roch, while the second is proved in \cite{fv1} which also contains (2) and (3). We briefly recall the Riemann-Roch argument:
Observe that the fibre of $\mgn/S_{n} \to \mg$ over a smooth curve $[C]\in \mg$ is birational to the symmetric product $C_{n}:=\Sym^n(C):= C^n/S_n$. This can be interpreted as the space of effective divisors of degree $n$ on the curve $C$. Since the Riemann-Roch theorem implies that any effective divisor of degree $d> g$ lies in some $\mathfrak{g}^1_d$, the quotient  $\mgn/S_{n}$ is trivially uniruled for $n> g$. 

Thus, outside the domain of values specified in Proposition \ref{one}, there is just a narrow transitional band in which $\mS$ changes from general type to its opposite.

It is, however, left open by Proposition \ref{one} if $\mG$ might be of general type for some values of $n>g$ if the subgroup $G$ of $S^n$ is chosen judiciously. In the following theorem we shall display a large class of groups for which this is the case.

\begin{theorem} \label{two}
Fix a partition $n=n_1 + \ldots + n_m$ and let $G=S_{n_1} \times \ldots \times S_{n_{m}}$.  Then $\mGgn$ (and thus, by Lemma \ref{sub}, $\mgn^H$ for any subgroup $H$ of $G$) is of general type if 
\begin{enumerate}
\item[(i)] $g \geq 24, \quad \max \{n_1, \ldots n_m \} \leq g-1$ or
\item[(ii)] $g \leq 23, \quad  \max \{n_1, \ldots n_m \} \leq g-2$ and
$f_m(g;n_1, \ldots n_m) \leq 13$, where $f_m$ is the function defined in equation \eqref{fm} of Section \ref{2} below.
\item[(iii)]  $g \leq 23, \quad  \max \{n_1, \ldots n_m \} \leq g-1$ and
$f_m(g;n_1, \ldots n_m, L_1, \ldots ,L_m) \leq 13$, where $f_m$ is the function defined in equation \eqref{fml} of Section \ref{2} below (depending on a choice of divisor classes 
$ L_1, \ldots ,L_m$  as described at the end of Section \ref{2}.
\end{enumerate}
Furthermore,  $\mG$ still has non-negative Kodaira dimension if $\max \{n_1, \ldots n_m \} \leq g$ and $f_m(g;n_1, \ldots n_m) \leq 13$.

\end{theorem} 

A geometric interpretation of this result is similar to the interpretation for $\cA_n$: The partition $P:  n=n_1 + \ldots + n_m$  induces the group 
$G=G_P=S_{n_1} \times \ldots \times S_{n_{m}}$, and the action of $G_P$ maps an $n$-pointed  genus g curve to a curve with markings in $n_1, \ldots, n_m$ (considered as an ordered $m-$ tuple) which we may call a $P -$ marked curve. Thus Theorem \ref{two} states that the moduli space of $P-$ marked genus $g$ curves is of general type (if the conditions in Theorem \ref{two} are satisfied).

Since there is no upper bound on the number of summands $m$ in the partition of $n$, an inspection of the defining equation for $f_m$ shows that the values of $n$ may tend to infinity, provided the subgroup $G$ is chosen appropriately. 
As in the above case for $ m=2, $ Riemann-Roch establishes a (small) transitional band beyond which $\mGgn$ becomes uniruled. We emphasize that the
existence of  this transitional band (for any {\em fixed} subgroup $G$) is different from the result for $\ngn$ proved in \cite{s2}: 
$\ngn=\mgnn/(\Z_2)^n\rtimes S_n$ is of general type for all values of $n$, if $g \geq 24$. This is perfectly compatible with Theorem \ref{two}, since $G:=(\Z_2)^n\rtimes S_n$ with its action on the $2n$ marked points is {\em not}
given by a direct product subgroup of $S_{2n},$  as required in Theorem \ref{two}. To understand this from a more conceptual point of view, however, is wide open at present. We emphasize that it is not merely the size of the group which is relevant: The alternating group $G=\cA_n$ might be taken arbitrarily large and still $\mG$ will preserve general type, while taking the quotient by much smaller groups, e.g. $G=S_{g+1}$ will turn $\mG$ to being uniruled.

As an application we consider $m=2$ and the special case $G=S_n \times S_n$. This quotient has a geometric interpretation as the universal difference variety, i.e. the fibre of the map $\mG \to \mg$ over a smooth curve $C$ is birational to the image of the difference map $C_n \times C_n \to J^0(C), (D,E)\mapsto D-E$,  see e.g. \cite{acgh}.

Then $\M^G_{g,2n}$ is uniruled for $n>g$ (by the Riemann-Roch argument from above, combined with the fact that a product is uniruled if one factor is), while Theorem \ref{two} gives the following.

\begin{proposition}\label{difvar}
The universal difference variety $\mgnn/S_n\times S_n$ is of general type for $g \geq 24, \quad n \leq g-1$, or, in the low-genus case, if $10 \leq g \leq 23$ and
$n_{\mathrm{min}}(g) \leq n \leq g-2$ where $n_{\mathrm{min}}(g) $ is specified in the following table

\begin{table}[h!]
$$\begin{array}{c|c|c|c|c| c|c|c|c|c| c|c|c|c|c}
g &10&11&12&13&14&15&16&17&18&19&20&21&22&23\\
\hline
n_{\mathrm{min}}&7&8&8&7&7&7&6&6&7&5&4&3&5&2
\end{array}$$
\caption{} \label{table dif}
\end{table}

\end{proposition}

This corollary amplifies the results of \cite{fv3}, which considers the universal difference variety in the special case $n=\lceil \dfrac{g}{2} \rceil$. We emphasize that the result in Table 3 for $g=13$ and $n=7$ is taken from \cite{fv3}; in view of the sharp coupling between $g$ and $n$ they are able to use in this special case an additional divisor, which is not applicable in the other cases and which is not contained in our Section \ref{div}. All other cases in Table 3 follow from our Theorem \ref{two}.

The outline of the paper is as follows. In Section \ref{prem} we introduce notation and some preliminary results, in Section \ref{div} we introduce the class of divisors used in our proof. Here we basically recall, for the sake of the reader, some material from \cite{s2}.  
In Section \ref{2} we prove Theorem \ref{two}. 
The use of a small program in computer algebra is appropriate to check our calculations.

\section{Preliminaries and Notation} \label{prem}

The aim of this section is to develop a sufficient condition for $\mS$ being of general type. This requires a basic understanding of the Picard group $\Pic( \mS)$ and an explicit description of the  boundary divisors and tautological classes on $\mS$ which we shall always consider as $S_n$-invariant divisors on $\mgn$ (any such divisor descends to a divisor on $\mS$).
For results on $\mgn$ we refer to the book \cite{acg} (containing in particular the relevant results from the papers \cite{ac1} and \cite{ac2}). All Picard groups are taken with rational coefficients and, in particular, we identify the the Picard group on the moduli stack with that of the corresponding coarse moduli space.

In particular, we recall the notion of the Hodge class $\l$ on $\mgn$, which automatically is $S_n-$ invariant and thus  gives the Hodge class $\l$ on $\mS$ (where, by the usual abuse of notation, we denote both classes by the same symbol; this abuse of notation is continued throughout the paper).

In order to describe the relevant boundary divisors on $\mgn$, we recall that $\Delta_0$ (sometimes also called $\Delta_{\mathrm{irr}}$) on 
$\mg$ is the boundary component consisting of all (classes of) stable curves of arithmetical genus $g$, having at least one nodal point with the property that ungluing the curve at this node preserves connectedness. Furthermore, $\Delta_i$,  for $1 \leq i \leq \lfloor\frac{g}{2}\rfloor,$ denotes the boundary component of curves possessing a node of order $i$ (i.e. ungluing at this point decomposes the curve in two connected components of arithmetical genus $i$ and $g-i$ respectively). Similarly, on $\mgn$ and for any subset $S \subset \{1, \ldots,n\}$, we denote by $\Delta_{i,S}, 0 \leq i \leq \lfloor\frac{g}{2}\rfloor,$ the boundary component consisting of curves possessing a node of order $i$ such that after ungluing the connected component of genus $i$ contains precisely the marked points labeled by $S$. Note that, if $S$ contains at most 1 point, one has $\Delta_{0,S}= \emptyset$ (the existence of infinitely many automorphisms on the projective line technically violates stability). Thus, in that case, we shall henceforth consider $\Delta_{0,S}$ as the zero divisor.

We shall denote by $\delta_i, \delta_{i,S}$ the rational divisor classes of $\Delta_i, \Delta_{i,S}$ in $\Pic \mg$ and $\Pic \mgn$, respectively. Note that $\delta_0$ is also called $\delta_{\mathrm{irr}}$ in the literature, but we shall reserve the notation $\delta_{\mathrm{irr}}$ for the pull-back of $\d_0$ under the forgetful map
$\pi:\mgn \to \mg$.

We write $\d$ for the sum of all boundary divisors and set $\d_{i,s}=\sum_{|S|=s}\d_{i,S}$.We remark that a single $\d_{i,S} $ is not $G-$invariant (for a subgroup $G$ of $S_n$), but the divisor $\sum_{g \in G} \d_{i,g(S)}$, averaged by the action of $G$, obviously is. In particular $\d$ and $\d_{i,s}$ are always $G$-invariant. We shall use such an averaging in the proof of Theorem \ref{two}.

Next we recall the notion of the point bundles $\psi_i, 1 \leq i \leq n,$ on $\mgn$. Informally, the line bundle $\psi_i$ (sometimes called the cotangent class corresponding to the label $i$) is given by choosing as fibre of $\psi_i$ over a point $[C;x_1, \ldots, x_n]$ of $ \mgn$ the cotangent line $T_{x_i}^v(C)$. For later use we also set
\begin{equation}  \label{basechangeprep}
\omega_i:= \psi_i - \sum_{S \subset \{1, \ldots,n \}, S \ni i} \d_{0,S},
\end{equation}
and introduce  $\psi=\sum_{i=1}^n \psi_i.$  Clearly,  
the class $\psi$ is $S_n-$ invariant.\\


As a first step in the direction of our sufficient criterion we need the following result on the geometry of the moduli space $\mG.$

\begin{theorem}  \label{noadcon}
For any subgroup $G$ of $S_n$, the moduli space $\mG$ has only {\em canonical singularities} or in other words:
The singularities of $\mG$ do not impose {\em adjunction conditions}, i.e. if 
$\r: \tmG  \to \mG$ is a resolution of singularities, then for any  $\ell \in \N$
there is an isomorphism
\begin{equation}
\r^*: H^0((\mG)_{\mathrm{reg}}, K_{(\mG)_{\mathrm{reg}}}^{\otimes \ell}) \to H^0(\tmG,K_{\tmG}^{\otimes \ell}).
\end{equation}
Here $(\mG)_{\mathrm{reg}}$ denotes the set of regular points of $\mG$, considered as a projective variety and $K_{\tmG},  K_{(\mG)_{\mathrm{reg}}}$ denote the canonical classes on $\tmG$ and  $(\mG)_{\mathrm{reg}}$.
\end{theorem}

The proof 
follows the lines of the 
the proof of Theorem 1.1 in \cite{fv1}. 
We shall briefly review the argument. A crucial input is Theorem 2 of the seminal paper \cite{hm} which proves that the moduli space $\mg$ has only canonical singularities. The proof relies on the Reid-Tai criterion: Pluricanonical forms (i.e. sections of $K^{\otimes \ell}$) extend to the resolution of singularities, if for any automorphism $\sigma$ of an object of the moduli space the so-called {\em age} satisfies  $age(\sigma) \geq 1$. The proof in \cite{fv1} then proceeds to verify the Reid-Tai criterion for the quotient of $\mgn$ by the full symmetric group $S_n $. Here one specifically has to consider those automorphisms of a given curve which act as a permutation of the marked points. For all those automorphisms the proof in \cite{fv1} verifies the   Reid-Tai criterion. Thus, in particular, the criterion is verified for all automorphisms which act on the marked points as an element of some subgroup of $S_n$. Thus, the proof in \cite{fv1} actually establishes the existence of only canonical singularities for {\em any} quotient $\mgn/G$ where $G$ is a subgroup of $S_n$. Clearly, this is our theorem.

Theorem \ref{noadcon} implies that the Kodaira dimension of $\mG$ equals the Kodaira-Iitaka dimension of the canonical class $K_{\mG}$. In particular, $\mG$ is of general type if  $K_{\mG}$ is a positive linear combination of an ample and an effective  rational class on $\mG$.  It is convenient to slightly reformulate this result. We need

\begin{proposition} \label{psi}
The class $\psi$ on $\mgn$ is the pull-back of a divisor class on $\mG$ which is big and nef.
\end{proposition}
 
\begin{proof}
 Farkas and Verra have proven in  Proposition 1.2 of \cite{fv2} that the $S_{n}$-invariant class $\psi$ descends to a big and nef divisor class $N_{g,n}$ on the quotient space $\mgn /S_{n}$. Consider the sequence of  natural projections
 $\mgn \xrightarrow{\pi} \mG \xrightarrow{\n} \mS.$
%
Then $\n^*(N_{g,n})$ is a big and nef divisor on $\mG=\mgn/G$ and $\pi^*(\n^*(N_{g,n})=\psi$. 
\end{proof}

 Now observe that the ramification divisor (class) of the quotient map $\pi:\mgn \to \mG$ is precisely 
\begin{equation}\label{R}
R = \sum_{(i,j) \in G} \d_{0, \{i,j\}}.
\end{equation} 
In fact ramification requires existence of a non-trivial automorphism belonging to $G$, and by standard results this only occurs in the presence of the projective line with 2 marked points that can be swapped. The non-trivial automorphism is then the transposition (of the labels) of these two marked points. 
 Furthermore, the Hurwitz formula for the quotient map $\pi$ gives
\begin{equation}\label{K}
K:= \pi^*(K_{\mG})= K_{\mgn} - R= 13 \l + \psi - 2 \d -\sum_{(i,j) \in G} \d_{0, \{i,j\}} .
\end{equation} 


 We thus obtain the final form of our sufficient condition:
 If  $K$ is a positive multiple of $\psi$ + some effective $G - $ invariant divisor class on $\mgn$, then $\mG$ is of general type.


\section{divisors}   \label{div}

In this section we introduce the relevant $S_n -$ invariant effective divisors on $\mgn.$ First we recall the following standard result.

\begin{proposition}  \label{divisor}
Let $f:X \to Y$ be a morphism of projective schemes, $D \subset Y$ be an effective divisor and assume that $f(X)$ is not contained in $D$. Then $f^*(D)$ is an effective divisor on $X$.
\end{proposition}
In our case the assumption of this proposition is fulfilled automatically since we only consider surjective maps.



We shall need invariant divisors on $\mgn$. Rather than exhibiting them directly by explicit definitions, we shall simply recall from the literature the existence of special divisors with small  {\em slope}:  If $g+1$ is not prime, then there is an effective $S_n-$invariant divisor class $D$ on $\mg$ (of Brill-Noether type) of slope
\begin{equation}  \label{sl1}
 s(D)=6 + \frac{12}{g+1},
\end{equation} 
while for $g+1$ odd (which trivially includes the case $g+1$ being prime) there is an effective $S_n-$invariant divisor class $D$ on $\mg$ (of Giesecker-Petri type) of slope
\begin{equation}  \label{sl2}
 s(D)=6 + \frac{14 g+4}{g^2+2g},
\end{equation} 
see \cite{eh}. For a few cases ($g=10, 12, 16, 21$) it has been shown in   \cite{fv4} (for g=12)  and \cite{f} (for the other 3 cases) that there exist special effective invariant divisors $D=D_g$  with even smaller slope, i.e.
\begin{equation} \label{sl3}
s(D_g) = 
\begin{cases}
7 & g=10 \\
6 + \frac{563}{642}  \qquad \qquad  & g=12 \\
6 + \frac{41}{61}  \qquad \qquad  & g=16 \\
6 + \frac{197}{377}  \qquad \qquad  & g=21 .
\end{cases}
\end{equation}
We shall need them in the proof of Theorem \ref{one}. 

Finally, we need divisors of Weierstrass-type, and these we have to introduce explicitly. We recall from \cite{l1}, Section 5,
the divisors $W(g;a_1, \ldots,a_m)$ on $\mgm$, where $a_i \geq 1$ and $\sum a_i=g$.
They are given by the locus of curves $C$ with marked points $p_1, \ldots, p_m$ such that there exists a $\mathfrak{g}^1_g$ on $C$ containing $\sum_{1 \leq i \leq m}a_i p_i$. We want to minimize the distance between the weights $a_i$. Thus we decompose
$g=km+r$, with  $r<m$,  and set
\begin{equation}\label{wmdef}
\tilde{W}_{g,m} =W(g;a_1, \ldots,a_m), \qquad a_j=k+1 \, \, (1\leq j \leq r), \quad  
a_j=k  \, \, (r+1 \leq j \leq m).
\end{equation}

This gives, in view of \cite{l1}, Theorem 5.4,
\begin{equation}\label{wm}
\begin{split}
\tilde{W}_{g,m} 
& = -\l+	\sum_{i=1}^r \frac{(k+1)(k+2)}{2}\omega_i +\sum_{i=r+1}^m \frac{k(k+1)}{2} \omega_i  -0\cdot\dirr 
\\&-\sum_{i,j\leq r} (k+1)^2 \d_{0,\{i,j\} }  -\sum_{i\leq r,j>r} k(k+1) \d_{0,\{i,j\} } -\sum_{i,j>r} k^2 \d_{0,\{i,j\} } \\
&-\mbox{higher order boundary terms},  
\end{split}
\end{equation}
where {\em higher order}  means a positive linear combination of $\d_{i,S}$ where either $i>0$ or $|S|>2$.

From $\tilde{W}_{g,m}$ we want to generate  a $G - $ invariant divisor class $\tilde{W}_{g,n,m}$ on $\mgn$, by summing over appropriate pullbacks. Thus we let $S,T$ be disjoint subsets of $\{1,\ldots,n \}$ with $|S|=r$ and $|T|=m-r$ (recall that $r$ is fixed by the decomposition   $g=mk+r$) and let 
\begin{equation}
\pi_{S,T}: \mgn \to \mgm
\end{equation}
be a projection (i.e. a surjective morphism of projective varieties) mapping the class 
$[C; q_1, \ldots ,q_{n}]$ to $[C; p_1, \ldots ,p_{m}],$
where the points $q_i$ labeled by $S$ are sent to the points $p_1, \ldots,p_r$ (all with weights $a_i=k+1$) and the points labeled by $T$ are sent to the points $p_{r+1}, \ldots,p_m$ (all with weights equal to $k$). Clearly, for fixed $g$, there are precisely  $\binom{n}{r} \binom{n-r}{m-r}$ such projections. With this notation, we introduce
\begin{equation}\label{W}
\begin{split}
\tilde{W}_{g,n,m} := &\sum_{S,T} \pi_{S,T}^* \tilde{W}_{g,m} \\
= &-w_\l \l +w_\psi \psi+0\cdot\dirr-\sum_{s\geq 2} w_s \d_{0,s}
\\&-\mbox{higher order boundary terms},  
\end{split}
\end{equation}
where  {\em higher order} denotes a positive linear combination of boundary divisors $\d_{i,S}$ with     $i \geq 1$,
\begin{equation}  \label{ws}
w_s \geq s w_\psi \geq 3w_\psi \qquad \mbox{for } s\geq 3,
\end{equation}
\begin{equation}  \label{w lambda}
w_\l= \binom{n}{r} \binom{n-r}{m-r},
\end{equation}
\begin{equation}\label{w psi}
 w_\psi=\binom{n-1}{r-1} \binom{n-r}{m-r} \dfrac{(k+1)(k+2)}{2} +\binom{n-1}{r} \binom{n-r-1}{m-r-1} \dfrac{k(k+1)}{2},
\end{equation}
\begin{equation}\label{w2}
\begin{split}
w_2= &2w_\psi + \binom{n-2}{r-2} \binom{n-r}{m-r}(k+1)^2 +2\binom{n-2}{r-1} \binom{n-r-1}{m-r-1}k(k+1) 
 \\&+\binom{n-2}{r} \binom{n-r-2}{m-r-2}k^2.  
\end{split}
\end{equation}

 Equation \eqref{w psi} is proved by applying pullback to \eqref{wm}, using $\o:=\sum_{i=1}^{n} \o_i$,
 \begin{equation}
\sum_{S,T} \sum_{i=1}^r  \pi_{S,T}^* \o_i = \binom{n}{r} \binom{n-r}{m-r} \frac{r}{n} \o =\binom{n-1}{r-1} \binom{n-r}{m-r}\o
 \end{equation}
and
\begin{equation}
\sum_{S,T} \sum_{i=r+1}^m  \pi_{S,T}^* \o_i = \binom{n}{r} \binom{n-r}{m-r} \frac{m-r}{n} \o=\binom{n-1}{r} \binom{n-r-1}{m-r-1},
\end{equation}
noting that equation \eqref{basechangeprep} implies
\begin{equation} \label{basechange}
\o=\psi - \sum_S |S| \d_{0,S}.
\end{equation}

The sums over the pullbacks of the boundary divisors are computed by similar combinatorial considerations which we leave to the reader. Note that both the summand
$2w_\psi$ on the right hand side of \eqref{w2}  and the bound in \eqref{ws} are  generated by the change of basis given in \eqref{basechangeprep}.\\
Next, for the proof of Theorem \ref{two} it will be convenient to renormalize the divisor $\tilde{W}_{g,n,m}$ in such a way that the coefficient of $\psi$ is equal to 1. We thus introduce  $W_{g,n}= \frac{1}{w_\psi} \tilde{W}_{g,m}$ and find, setting $m= \min \{g,n \}$, 
\begin{equation}
\label{wgn}
W_{g,n}  
= a(g,n) \l  +  \psi+ 0\cdot\dirr- \sum_{s\geq 2} b_s \d_{0,s}
-\mbox{higher order boundary terms},  
\end{equation}
where  {\em higher order} denotes a positive linear combination of boundary divisors $\d_{i,S}$ with     $i \geq 1$,
\begin{equation}
a(g,n)=
\begin{cases}  

\frac{2n}{(k+1)(g+r)} & g=kn+r, r<n \\
\frac{n}{g} &  n>g    
\end{cases}
\end{equation}
$b_2=b(g,n)$  with
\begin{equation}
b(g,n)= \begin{cases}
2+ \frac{2}{n-1}\frac{r(r-1)(k+1)^2 + 2r(n-r)k(k-r)+(n-r)(n-r-1)k^2}{r(k+1)(k+2)+(n-r)k(k+1)} \qquad \qquad & g=kn+r,r<n   \\
2 + \frac{g-1}{n-1} & n>g
\end{cases}
\end{equation}
and  $b_s > b_2$ for all $s>2$.

In addition, we shall use the anti ramification divisor classes from \cite{fv2}, Section 2, to obtain (by straightforward though somewhat lenghty algebraic computation) the existence of effective divisor classes $T_g$ on $\M_{g,g-1}$ satisfying
\begin{equation}  \label{tg}
T_g = - \frac{g-7}{g-2} \lambda + \psi - \frac{1}{2g-4} \dirr - \left( 3 + \frac{1}{2g-4} \right) \delta_{0,2} + h.t.
\end{equation}
 where the higher order terms $h.t.$ denote a linear combination of all other boundary divisors with coefficients $\leq -2$.
 
Furthermore, normalizing the divisor classes in \cite{fv2}, Theorem 3.1, one obtains for $g \geq       1$ and any $1 \leq m \leq g/2$ effective divisor classes $F_{g,m}$ on $\mgn$ (with $n=g-2m$)  satisfying
\begin{equation}   \label{fgm}
F_{g,m} = a \lambda + \psi - b_{irr} \dirr -   b_{0,2} \delta_{0,2} + h.t.,
\end{equation}
where, as above,  the higher order terms $h.t.$ denote a linear combination of all other boundary divisors with coefficients $\leq -2$ and
\begin{equation}
a= \frac{n}{n-1}\left( \frac{10m}{g-2}+\frac{1-g}{g-m} \right), \quad b_{0,2}=3+\frac{(g-n)(n+1)}{(g+n)(n-1)}, \quad b_{irr}= \frac{nm}{(g-2)(n-1)}
\end{equation}
Finally, to cover the case where $g$ and $n$ have different parity, we set $n=g-2m+1$ and pull back $F_{g,m}$ given in eqution \ref{fgm} in all possible ways to $\mgn$ (via a forgetful map forgetting one of the marked points). Summing all these divisor classes and then normalizing gives an effective divisor class $\tilde{F}_{g,m}$ on $\mgn$ satisfying
\begin{equation}   \label{tfgm}
\tilde{F}_{g,m} = a \lambda + \psi - b_{irr} \dirr -   b_{0,2} \delta_{0,2} + h.t.,
\end{equation}
where, as above,  the higher order terms $h.t.$ denote a linear combination of all other boundary divisors with coefficients $\leq -2$ and
\begin{equation}
a= \frac{n}{n-2}\left( \frac{10m}{g-2}+\frac{1-g}{g-m} \right), \quad b_{0,2}=3+\frac{g-n-1}{g+n-1}, \quad b_{irr}= \frac{nm}{(g-2)(n-2)}.
\end{equation}

\section{Proof of Theorem \ref{two}}  \label{2}

For  assertion (i), recall that $\mg $ is of general type for $g \geq 24.$ Furthermore, the generic fibre of the  canonical projection $ \mG \to \mg$ is
$$ C^n /G  \simeq    (C^{n_1}/S_{n_1}) \times \ldots \times (C^{n_m}/S_{n_m}).$$
For $\max\{n_1, \ldots ,n_m\} \leq g-1$ each factor $C^{n_i}/S_{n_i}$ is of general type, and thus is the product $ C^n /G.$ Therefore the assertion follows from weak additivity of the Kodaira dimension.

Assertion (ii) is more complicated, involving divisors. We take  Weierstrass divisors $W_k=W_{g, n_k}$ on each $\M_{g,n_k}$ and $W=W_{g,n}$ on $\mgn$, with coefficients $a(g,n_k),b(g,n_k)$ and $a(g,n),b(g,n)$ respectively, see Section \ref{div} \eqref{wgn}.  Let $S_k$ be the set of points corresponding to the summand $n_k$ in the partition $n=\sum_{1 \leq k \leq m} n_k$, and denote by $\pi_k: \mgn \to \M_{g,n_k}$ the forgetful map forgetting all points  except those in $S_k$. In order to calculate $\pi_k^* W_k$ we introduce some notation.

For any sets $S \subset \{1, \ldots,n\}$ we define
\begin{equation}
\d_{i,s}^{S,\ell} := \sum_{|T \cap S|=\ell, |T|=s} \d_{i,T}
\end{equation}
and denote by $\pi_S: \mgn \to \M_{g,|S|}$ the natural forgetful map. With this notation, by the usual abuse of notation amplified in Section \ref{prem}, we have 

\begin{proposition}  \label{forgetful}

The pull-back divisors are 
$\pi^*_S(\l)=\l, \quad \pi^*_S(\dirr) =\dirr$ and
\begin{equation}
\pi^*_S(\psi) = \sum_{i \in S} \psi_i - \sum_{s=2}^{n-n_k +1} \d_{0,s}^{S,1}, 
\qquad \qquad  \pi^*_S(\d_{i,s}) = \sum_{\ell \geq 0} \d_{i,s+\ell}^{S,s}.
\end{equation}
\end{proposition}

Furthermore, observe that the labels $i,j$ belong to different components $S_k,S_\ell$ if and only if the transposition $(i\ j)$ is not in $G$. This gives:  the divisor $L:= \sum_{1 \leq k \leq m} \pi^*_k W_k$ has the decomposition
\begin{equation}  \label{L}
L= - \sum_{1 \leq k \leq m} a(g,n_k) \l + \psi -2 \sum_{(i\ j)\notin G} \d_{0,\{i,j\}} + 0 \dirr - \sum_{1 \leq k \leq m} b(g,n_k) \sum_{i,j \in S_k} \d_{0,\{i,j\}} + h.t,
\end{equation}
where $h.t.$ denotes a (higher order) sum of boundary divisors, each multiplied
with coefficients $<-2$.
In addition we consider
\begin{equation}  \label{WD}
W= -a(g,n) \l + \psi - b(g,n) \sum_{i,j} \d_{0,\{i,j\}} + h.t, \qquad D=s\l -\dirr + h.t.,
\end{equation}
where $D=D_g$ is chosen with minimal slope $s=s(g)$ (see the list of divisors with small slope in \eqref{sl1}-\eqref{sl3}) and set 
\begin{equation} \label{e}
\epsilon := \min \{ b(g,n_k) -3| \quad k \in \{1, \ldots, m \}  \mbox{  with  } n_k \geq 2 \}.
\end{equation}

Clearly, $ \epsilon >0$ if and only if $\max \{n_1, \ldots n_m\} \leq g-2$. 
Combining equations \eqref{L},  \eqref{WD}, \eqref{e} (see also \ref{wgn}) one obtains the decomposition
\begin{equation} \label{decompos}
K_G  \geq 2 D + \frac{1}{1+\epsilon} L + \frac{2 \epsilon}{b(g,n)(1+\epsilon)}W + \eta \psi, \qquad \eta:= \frac{\epsilon}{1+\epsilon}(1-\dfrac{2}{b(g,n)})>0,
\end{equation}
provided one has the inequality
\begin{equation}  \label{fm}
f_m(g,n_1, \ldots,n_m):= 2 s(g) - \frac{1}{1+\epsilon} \sum_{1 \leq k \leq m} a(g,n_k) - \frac{2 \epsilon}{b(g,n) (1+ \epsilon)} a(g,n)  \leq 13.
\end{equation}

Since by Proposition \ref{psi} the divisor  class $\psi$ is big and nef and 
all divisors in equation \eqref{decompos} are effective and $S_n -$invariant, the proof boils down to checking the inequality \eqref{fm}. 

Note that for $\max \{n_1, \ldots, n_m\}\in \{g-1,g\}$ we get $\epsilon=\eta=0$, which proves that $K_G$ is at least effective and thus gives non-negative Kodaira dimension.

To treat the additional case $\max \{n_1, \ldots n_m \} = g-1$ in case (iii) we need more general divisors $L_1, \ldots ,L_m$. The function $f_m$ in equation \eqref{fm} will then depend on these divisors, destroying the explicit form of $f_m$ given  in equation \eqref{fm}. 

Instead of the family of (generalized) Weierstrass divisores $W_k$ on $\M_{g,n_k}$, for $1 \leq k \leq m,$ we use divisors $L_k$  on $\M_{g,n_k}$, for $1 \leq k \leq m,$ having a decomposition
\begin{equation}  \label{lk}
L_k=a_k \lambda+\psi-b_{k,irr} \dirr -b_k \delta_{0,2} + h.t.
\end{equation}
where $b_k >3$ and  the higher order terms $h.t.$ denote a linear combination of all other boundary divisors with coefficients $\leq -2$. Setting (analog to the above)  
$L:= \sum_{1 \leq k \leq m} \pi^*_k L_k$ we obtain
\begin{equation}  \label{Lnew}
L= - \sum_{1 \leq k \leq m} a_k \l + \psi - 2 \sum_{(i\ j)\notin G} \d_{0,\{i,j\}} -  
\sum_{1 \leq k \leq m} b_{k,irr} \dirr - \sum_{1 \leq k \leq m} b_k \sum_{i,j \in S_k} \d_{0,\{i,j\}} + h.t,
\end{equation}
with $h.t.$ as above. This is analog to \eqref{L}.

In this notation, we already have for shortness's sake suppressed dependence on $g,n$. Using the same convention in equation \eqref{WD} - thus simply writing $a,b$ in the decomposition of $W$ - and introducing
\begin{equation} \label{enew}
\epsilon := \min \{ b_k -3| \quad k \in \{1, \ldots, m \}  \mbox{  with  } n_k \geq 2 \},
\end{equation}
which is \eqref{e} with $b(g,n_k)$ replaced by $b_k$ and writing $\a_+:=\max\{\a,0\}$, we obtain the decomposition
\begin{equation} \label{decomposen}
K_G  \geq (2- \frac{1}{1+\epsilon}\sum_{k} b_{k,irr})_+  D + \frac{1}{1+\epsilon} L + \frac{2 \epsilon}{b(1+\epsilon)}W + \eta \psi, \quad \mbox{where} \quad \eta:= \frac{\epsilon}{1+\epsilon}(1-\frac{2}{b})>0,
\end{equation}
provided one has the inequality

\begin{equation}  \label{fml}
f_m(g,n_1, \ldots,n_m,L_1, \ldots ,L_m):= (2- \frac{1}{1+\epsilon}\sum_{k} b_{k,irr})_+ s + \frac{1}{1+\epsilon} \sum_{1 \leq k \leq m} a_k - \frac{2 \epsilon}{b (1- \epsilon)} a  \leq 13.
\end{equation}

This finishes the proof.

\section{Proof of Proposition \ref{difvar}}  \label{last}

In case $n \leq g-2$ and $n_{\mathrm{min}}(20)=5$ (instead of 4) and $n_{\mathrm{min}}(22)=6$ (instead of 5) this is a direct application of Theorem \ref{2}, case (ii).

To cover the remaining cases, we shall use Theorem \ref{2}  in case case (iii) with the following choice of divisors $L_k$ for $1 \leq k \leq 2$ (clearly $m=2$ for the difference variety) in equation \eqref{fml}.

In case $n=g-1$, we choose $L_1=L_2=T_g$, defined in equation \eqref{tg}. It is straightforward to check that $f_m \leq 13$ in this case.

In case $g=20$ and $n=4$, we choose $L_1=L_2=F_{20,8}.$ Again, since now all divisor classes are explicit, it is straightforward to check  $f_m \leq 13$. Finally, in case
 $g=22$ and $n=5$, we choose $L_1=L_2=\tilde{F}_{22,9}.$

\end{document}